\documentclass[a4paper]{amsart}
\usepackage[ansinew]{inputenc}        
\usepackage{amsmath,amsfonts,amssymb,amsthm}
\usepackage{tikz}
\newtheorem*{thm*}{Theorem}

\newtheorem{lemma}{Lemma}
\newtheorem{lemma*}{Lemma}
\newtheorem*{pro*}{Proposition}
\newtheorem{pro}{Proposition}

\newtheorem{cor}{Corollary}

\theoremstyle{definition}
\newtheorem{example}{Example}

\renewcommand {\Im}{\mathop{\mathrm{Im}}\nolimits}
\renewcommand {\Re}{\mathop{\mathrm{Re}}\nolimits}

\newcommand {\Log}{\mathop{\mathrm{Log}}\nolimits}

\author{Bernd Martin}
\address{Institute of Mathematics, Brandenburg University of Technology Cottbus, PF 101344, 03013 Cottbus   Germany}
\email{martinb@tu-cottbus.de}

\author{Dmitry Pochekutov}
\address{Institute of Core Undergraduate Programmes, Siberian Federal University,  av. Svobodny 79, Krasoyarsk, 660041 
    Russia}
\email{ potchekutov@gmail.com}
\thanks{supported by the research school 'Hybrid Systems' of BTU Cottbus}
\title{Discriminant and Singularities of  Logarithmic Gauss Map, Examples and Application}
\date{}
\begin{document}
\maketitle

\begin{abstract}
The study of hypersurfaces in a torus leads to the beautiful 
zoo of amoebas and their contours, whose possible configurations are seen from combinatorical data. 
There is a deep connection to the logarithmic Gauss map and its critical points. 
The theory has a lot of applications in many directions. 

In this report we recall basic notions and results from the theory of amoebas, 
show some connection to algebraic 
singularity theory and discuss some consequences from the well known 
classification of singularities to this subject.
Moreover, we have tried to compute some examples using the 
computer algebra system {\sc Singular} and discuss different 
possibilities and their effectivity to compute the critical points.
Here we meet an essential obstacle: 
Relevant examples need real or even rational solutions, 
which are found only by chance.
We have tried to unify different views to that subject.
\end{abstract}

\section{Toric hypersurface and Logarithmic Gauss map}

Let $V^*(f)$ be an algebraic hypersurface in the algebraic torus $\mathbb{T}^n$,
$\mathbb{T}:=\mathbb{C}^*$, i.e.
$$
	V^*(f)=\{z\in\mathbb{T}^{n}\,| \, f(z)=0\},
$$
where $f(z)$ is the Laurent polynomial. 

Recall that the \textit{Newton polyhedron} 
${\mathcal N}_f \subset \mathbb{R}^n$ of $f$ is the convex hull in $\mathbb{R}^n$ 
of $A_f:=\textup{supp}(f)\subset \mathbb{Z}^n$. 
 Let $\mathbb{X}_\Sigma$ be the smooth toric variety associated to the  fan $\Sigma$, which 
 is a refinement of the fan
 dual to the Newton polyhedron ${\mathcal N}_f$. 
We denote by $\overline{V}(f)\subset\mathbb{X}_\Sigma$ 
the closure of $V^*(f)$ in $\mathbb{X}_\Sigma$. The polynomial $f$ is called 
\textit{non-singular for its Newton polyhedron}, if $V^*(f)$ is smooth 
and for any face $\Delta\subset \mathcal{N}_f$ 
the truncation $f^{(\Delta)}$ of $f$ to the face $\Delta$ has 
non-vanishing Jacobian at all $z\in\overline{V}(f)\cap \mathbb{X}_\Delta$:
$$ (z_1\partial f^{(\Delta)}/\partial z_1, \ldots , 
z_n\partial f^{(\Delta)}/\partial z_n)\neq 0.$$
In accordance with singularity resolution theorem,   cf. \cite[page~291]{Ho},
a generic polynomial  $f$ is non-singular for its Newton  polyhedron and then $\overline{V}(f)$ is non-singular. 

Next we introduce the so-called \textit{logarithmic Gauss map} 
$\gamma_f:V^*(f)\to \mathbb{P}^{n-1}$.
Let $ {\mathbf t}^n$ denote the Lie algebra of ${\mathbb T}^n$,  
which is identified with the tangent space of ${\mathbb T}^n$ at the unit point ${\bf e}$.
For any point $z\in V^*$ shift the tangent space $T_z(V^*)$ by the torus multiplication 
(with $z^{-1}$) to a hyperplane  $h_z\subset{\mathbf t}^n$, inducing a point 
in the projective space of the dual $ {\mathbf t}^{n*}$, which we define to be $\gamma(z):=h_z^* \in
\mathbb{P}^{n-1}:=\mathbb{P}({\mathbf t}^{n*})$.
In coordinates of ${\mathbb T}^n$ the map $\gamma_f$ is given by 
$$ 
	\gamma_f(z)=\left(z_1 f_{z_1}:\ldots :z_n f_{z_n}\right)\in \mathbb{P}^{n-1}.
$$
Described in more geometric terms we have: Let $U\subset \mathbb{T}^n$ 
be a neighbourhood of a regular point $z$ on $V^*(f)$. 
Choose a branch of the logarithmic map (restricted to $U$) 
$ \log:U\to \mathbb{C}^n$, then the 
direction of the normal line at $\log(z)$ to transformed hypersurface $\log(V^*(f)\cap U)$
has components $\left(z_1 f_{z_1},\ldots ,z_n f_{z_n}\right)$. This construction 
does not depend on the choice of the branch of $\log$.

In \cite{Mi1}, 3.2, one can find the idea of a construction, how to extend  $\gamma_f$ 
in the  non-singular case to a finite map.
$$
	\overline{\gamma}_f: \overline{V}(f)\to \mathbb{P}^{n-1}.
$$ 
 Having a finite map $\gamma$ to a smooth variety, one can associate the {\em
 ramification locus} or the {\em  discriminant} as image of the critical locus:
 $\mathcal{D}:=\gamma({\mathcal{C}_\gamma})$, which is usually a hypersurface. 
 An analytic structure which is compatible with base chance was introduced by Teissier, cf. \cite{Tes}:
 The structure sheaf ${\mathcal O_D}$ is defined to be the quotient by the $0$-th 
 fitting ideal of $\gamma_*({\mathcal O_C})$.  
 In local coordinates the defining equation is obtained as the (classical)
 discriminat of the polynomial, that generates the finite extension of the structure sheafs
 over an open affine subsets.

From the well-known theorem of Kouchnirenko, cf. \cite[Th.~3]{Kou},
Mikhalkin obtains:
	\begin{pro}[\cite{Mi1}]
	\label{lemma:1} 
	If the polynomial $f$ is non-singular for its Newton polyhedron, then
	the degree of $\overline{\gamma}_f$ is obtained as $$\deg(\overline{\gamma_f})=
	n!\cdot \textup{Vol}({\mathcal N}_f).$$
	\end{pro}
For later calculation we give a description of the logarithmic Gauss map 
$\gamma_f$ in local coordinates.
Since $V^*(f)$ is smooth, we may assume w.l.o.g. 
$f_{z_n}:=\partial f/\partial z_n\neq 0$ locally, then  
a function $g(z')$, $z':=(z_1,\ldots ,z_{n-1})$, exists, such that $f(z',g(z'))\equiv 0$.
Hence, $g_{z_i}=-f_{z_i}(z',g)/f_{z_n}(z',g)$ and $(\log(g(z'))_{z_i}=g_{z_i}/g$ hold, and
one obtains the formula
	\begin{equation}
	\label{eq:a}
		\gamma_f(z')=\left(-z_1\frac{\partial \log g(z')}{\partial z_1}:\,	
		\ldots\, :-z_{n-1}\frac{\partial \log g(z')}{\partial z_{n-1}}:1\right).
	\end{equation} 

Then the fiber $\gamma^{-1}_f(y),\ y=(y_1:\ldots:y_n)\in \mathbb{P}^{n-1}$ is given by the zeros of the
local complete intersection ideal generate by $f$ and the 2-minors of 
$$\left(\begin{array}{ccc} z_1f_{z_1}& \ldots & z_nf_{z_n}\\ y_1 & \ldots & y_n\end{array}\right),$$
i.e. (in case of $y_n\neq 0$) $\gamma^{-1}(y)$ is defined by the complete intersection ideal 
\begin{equation}
	\label{eq:b}
	I_y:=(f,h_1,\ldots h_{n-1}),
\end{equation}	 	
where $h_i=y_n z_i f_{z_i}-y_i z_n f_{z_n}$. There are at most $n!\cdot\textup{Vol}({\mathcal N}_f)$
zeros in the torus by Proposition~1.

\section{Amoeba and its Contour versus Laurent series}

Consider a rational function $F(z)=h(z)/f(z)$ of $n$ complex variables and different 
Laurent expansions 
	\begin{equation}
	\label{eq:1}
		\sum_{\alpha\in \mathbb{Z}^n} c_\alpha z^\alpha
	\end{equation}
of $F$ centered at $z=0$. The most natural way to describe these expansions 
uses the amoeba of polar hypersurface $V^*=V^*(f)$. 
 
Recall, that the \textit{amoeba} $\mathcal{A}_{V^*}$ of a toric hypersurface
$V^*=V^*(f)$ is the image of $V^*$ by the logarithmic map 
$\textup{Log}:\mathbb{T}^n\to\mathbb{R}^n$, 
	$$
		\textup{Log}: (z_1,\ldots, z_n) \mapsto (\log|z_1|,\ldots, \log|z_n|).
	$$

The complement $\mathbb{R}^n -\mathcal{A}_{V^*}$ to the amoeba consists 
of a finite number of connected components~$E_i$, 
which are open and convex, cf.~\cite[Section~6.1]{Gelfand}.
These components are characterized in the following Proposition, 
 which is a summary of 
Propositions~2.5, 2.6 in \cite{FPT}, Theorem~10 and Corollary~6 in \cite[Section~I.5]{Ru}.
	\begin{pro}
	\label{thm:1}
		There exists an open subset $U_\mathcal{N}$ in the set of polynomials with fixed Newton 
		polyhedron
		$\mathcal{N}$ exists that satisfies  the following property:
		
		If $f\in U_\mathcal{N}$, then   
		there is a bijection
		from the set of lattice points of $\mathcal{N} \cap\mathbb{Z}^n$ to the set of components of
		$\mathbb{R}^n -\mathcal{A}_{V^*(f)}:\ 
		\nu \mapsto E_\nu$,
		such that the normal cone 
		$C^{\vee}(\nu)$ to $\mathcal{N}_f$
		at the point $\nu$ is the recession cone of the component $E_\nu$. 
	\end{pro}
A recession cone is the maximal cone which can be put inside $E_\nu$ by a translation.
If $f\not\in U_\mathcal{N}$, the expected component $E_\nu$ may not exist for some non-vertice 
lattice points $\nu\in\mathcal{N}$,
 because the associated Laurent series below does not converge. 

Given a component $E_\nu$ one obtains a Laurent 
series of $F$ centered at $z=0$ using the 
term $a_\nu z^\nu$ of~$f$ as denominator in a corresponding geometric progression  
	\begin{equation}
	\label{eq:2}
		\frac{1}{f}=\sum_{k=0}^\infty 
		\frac{\left(a_\nu z^\nu-f\right)^k}{(a_\nu z^\nu)^{k+1}}.
	\end{equation}
The set $\{\textup{Log}^{-1}(E_\nu)\}$ contains the domain of convergence for 
this Laurent series.
The support of expansion (\ref{eq:2}) is the minimal cone $K_\nu$, which after a translation 
by $\nu$ contains the face $\Delta$ of 
$\mathcal{N}_f$, which has $\nu$ as interior point.

A non-zero vector $q\in\mathbb{Z}^n\cap K_\nu$ defines a so-called
\textit{diagonal subsequence} $\{c_{k\cdot q}\}_{k\in\mathbb{N}}$ of the set of coefficients 
of  expansion (\ref{eq:1}). We will discuss its asymptotic in the next section.

The set of critical values of the map $\textup{Log}$ restricted to  $V^*$ is called 
the \textit{contour} $\mathcal{C}_{V^*}$ of the amoeba $\mathcal{A}_{V^*}$ (see \cite{PT}). 
The contour is closely related to the logarithmic Gauss map $\gamma_f$.
Recall Lemma~3 from \cite{Mi1}.
	\begin{lemma}
	\label{lm:2}
		The preimage of the real points under the logarithmic Gauss map is mapped by $\textup{Log}$
		to the contour:
	 	$$\mathcal{C}_{V^*}=\textup{Log}\left(\gamma_f^{-1}(\mathbb{P}^{n-1}_\mathbb{R})\right).$$
	\end{lemma}
	\begin{proof}
	Let $z$ be a regular point on ${V^*}$ and $U$ its neighbourhood. Since the map 
	$\left.\textup{Log}\right|_{V^*}$ is a composition of  $\log: z \mapsto (\log(z_1),\ldots ,\log(z_n))$ 
	and the projection $\Re:\mathbb{C}^n\to\mathbb{R}^n$, the point $z$ is critical for 
	$\left.\textup{Log}\right|_{V^*}$ if the projection $d\Re: T_z \log\,(V^*\cap U)\to \mathbb{R}^n$ 
	is not surjective at $z$.  
A fiber $T_z \log\,({V^*}\cap U)$ of the tangent bundle of the image by $\log$ of the hypersurface 
$V^*$ is the hyperplane $$\{ t\in \mathbb{C}^n: \left<\gamma_f(z),t\right>=0\}.$$ 
For real $\gamma_f(z)$ the projection $d\Re$ is not surjective.
If $\gamma_f(z)$ is not real one can
consider $\left<\gamma_f(z),t\right>=0$ as a system of linear equations  
with fixed real part $\Re t,$ and solve it with respect to 
$\Im t.$ Hence, $z$ is not  critical for $\left.\textup{Log}\right|_{V^*}.$
\end{proof}
Therefore, the contour $\mathcal{C}_{V^*}$ can be computed as the $\textup{Log}$-image of 
the zeros of the ideal  
	\begin{equation}
	\label{eq:cv}
	(f,q_n z_1 f_{z_1}-q_1 z_n f_{z_n},\ldots, q_n z_{n-1} f_{z_{n-1}}-q_{n-1} z_n f_{z_n}),
	\end{equation}
where $q$ runs through all real points $(q_1:\ldots:q_n)\in \mathbb{P}_\mathbb{R}^{n-1}$, (here w.l.o.g $q_n\neq 0$).

\section{Singularities of Phase Function}

Consider the function 
	$$
		\Phi:\mathbb{P}^{n-1}\times V^*\longrightarrow \mathbb{C},\ 
		\Phi(y,z)=\left< y, \log z\right> .
	$$
Introduce the \textit{phase function}  $\varphi_q:=\Phi(q,-)$, later we show that it is indeed a phase function
of some oscilllating integral. Denote by $\textup{Crit}(\varphi_q)\subset V^*$
 the set of critical (or stationary) points of function $\varphi_q$. It coincides with the preimage of the  logarithmic Gauss map $\gamma_f$:

\begin{pro} \label{thm:3}
	
The relative critical locus of $\Phi$ coincides with the graph of $\gamma_f$:  
$$ \textup{Crit}_{\mathbb{P}^{n-1}}(\Phi)=\Gamma_{\gamma_f}.
$$
\end{pro}
\begin{proof}
Assume $f_{z_n}\neq 0$, then we use local coordinates $z'$ on $V^*$, and consider the function $g(z')$ such that
 $f(z',g(z'))\equiv 0$.
	We obtain
$$
	\partial\Phi(z,y)
	/\partial z_i=\frac{y_i}{z_i}+\frac{y_n}{g(z')} \partial g(z')
	/\partial {z_i},
	\ i=1,\ldots, n-1.
$$ 
Up to a non-zero constant multiple the components of the gradient $\partial \Phi(z,y)/\partial z$ 
together with the defining polymonial $f(z)$ of $V^*$ 
give us the defining ideal (\ref{eq:cv}) of the fiber of $y$ by the logarithmetic Gauss map $\gamma$. 
\end{proof}
The last statement shows us that the $\textup{Log}$-image of $\textup{Crit}(\varphi_q)$ is contained in 
the contour $\mathcal{C}_{V^*}$ of the amoeba $\mathcal{A}_{V^*}$, and 
the tangent hyperplane to $\mathcal{C}_{V^*}$ at a point $\textup{Log}(z_0)$, $z_0\in  \textup{Crit}\,\varphi_q$,
has normal vector~$q\in\mathbb{Z}^n-\{0\}$.

Another consequence of the above formula concerns the connection between the
singularities in the fibers of the phase function and the fibers of the logarithmic Gauss
map:
\begin{pro}\label{thm:4}
 Let $(z_0,y_0)\in \Gamma_{\gamma_f}$ be a point of the graph of $\gamma_f$, 
 then the Jacobian matrix of $\gamma_f$
 at $z_0$ coincides with the Hesse matrix of $\varphi_{y_0}$ at $z_0$ up to
 multiplication with a regular constant diagonal matrix $D$:
 $$ \textup{Hess}(\varphi_{y_0})(z_0) =D \cdot \textup{Jac}(\gamma_f)(z_0). 
 $$   
\end{pro}
\begin{proof}
As before we assume $f_{z_n}(z_0)\neq 0$ and use local coordinates $z'$.
From~(\ref{eq:a}) we obtain the entries of the Jacobian matrix $\textup{Jac}(\gamma_f)$ of the map $\gamma_f$ 
	$$
		\textup{Jac}(\gamma_f)_{(i,j)}=-\left( z_i \frac{\partial^2 \log (g(z'))}{\partial z_i\partial z_j}+ 
		\delta_{ij}\frac{\partial\log(g(z'))}{\partial z_j}\right), i,j=1,\ldots, n-1,
	$$
	where $\delta_{ij}$ is the Kronecker symbol.
	Moreover, 
	$$
		\frac{\partial^2 \varphi_{y}}{\partial z_i\partial z_j}=y_{n}\frac{\partial^2 \log(g(z'))}
		{\partial z_i \partial z_j}
		-\delta_{ij}\frac{y_{i}}{z_i^2}
	$$	
	holds for the second derivatives of $\varphi_y$. 
	Since 
	$\displaystyle \frac{y_{0,i}}{z_{0,i}}=-y_{0,n} \frac{\partial \log(g(z'_0)) }{\partial z_i}$ at 
	a critical point $z_0$ of $\varphi_{y_0}$, we obtain the statement by putting 
	the $i$-th entry of $D$ to be $\displaystyle d_i=-\frac{y_{0,n}}{z_i}$.
\end{proof}

We obtain as corollary of the last proposition that for directions $y=q$ outside the ramification locus of
the logarithmic Gauss map $\gamma_f$ the phase function $\varphi_q$ has only Morse critical points.

\begin{cor}
	The logaritimic Gauss map $\gamma_f$ is umramified at $q\in  \mathbb{Z}^n-\{0\}$ iff
	 the phase function $\varphi_q$ has only  Morse
	critical points, e.g. non-degenerated singularities.
\end{cor}
\begin{proof}
The map $\gamma_f$ is not ramified over $y$ iff its Jacobian has full rank at all points of the fiber 
$\gamma^{-1}(y)$.
From Proposition~3 we get
	$$
		\det(\frac{\partial^2 \varphi_q}{\partial z_i\partial z_j}(z_0))
		=\frac{q_n^{n-1}}{z_{0,1}\cdots z_{0,{n-1}}}\det(Jac(\gamma_f)(z_0)).
	$$
Hence, the Jacobian determinant does not vanish iff the Hessian is not zero at corresponding points.	
\end{proof}
Next we want to discuss degenerated critical points of the phase function. 
By Mather-Yao type theorem the ${\mathcal R}$-class (right-equivalence)
of an analytic function $h(z)\in \mathbb{C}\{z\}=:{\mathcal O}_n$ at an isolated critical point $z_0=0$ 
is equivalent to the isomorphy type of the Milnor algebra $Q_h:={\mathcal O}_n/(\partial h/\partial z)$, 
but as $\mathbb{C}[t]$-algebra, the action of $t$ on $Q_h$ induced by multiplication with $h$, cf. \cite{BM}. 
The isomorphy type of the associated singularity $(V(h),z_0)$, i.e. the ${\mathcal K}$-class (contact-equivalence) of $h(z)$, is equivalent to the isomorphy class of the Tjurina algebra $T_h=Q_h/(h)$ itself, cf. \cite{Yau}.  
Obviously, these equivalence classes coincide for quasi-homogeneous functions 
(because $\mu(h)=\tau(h)$, $T_h=Q_h$, $hQ=0$).
The Milnor algebra of the phase function at $z_0$ coincides with the local algebra of the fiber  
$\gamma_f^{-1}(\gamma_f(z_0))$ at $z_0$.
\begin{cor}
 If $(z_0,y_0)\in \Gamma_{\gamma_f}$, denote by  
 $Q_\varphi$ the Milnor algebra of the function $(\varphi_{y_0}(z)-y_0)$ at $z_0$, then we have
 $$Q_\varphi={\mathcal O}_{\gamma_f^{-1}(y_0),z_0} \ \ \mbox{and}\ \
 Q_\varphi/Ann(\mathbf{m}_Q) = {\mathcal O}_{Sing(\gamma_f^{-1}(y_0)),z_0}.$$
\end{cor}
\begin{proof}
By Proposition 3 
the germs coincide: 
$(Crit(\varphi_{y_0}),z_0) = (\gamma_f^{-1}(y_0),z_0)$. The algebra of the critical
locus is the Milnor algebra of  $(\varphi_{y_0}(z)-y_0)$. By Proposition 4 
the Jacobi determinant of
$\gamma_f$ at $z_0$ equals up to a constant multiple to the Hessian of $\varphi_{y_0}$ at $z_0$, 
which generates the 
annulator of the maximal ideal in the local complete intersection algebra $Q_\varphi$.
\end{proof} 
A function $h\in {\mathcal O}_n$ with isolated critical point is called {\em almost quasihomogeneous}, if $\mu=\tau+1$.
This is equivalent to $hQ_h=Ann({\bf m}_Q)$.
Assume that the singularities in a fiber of a phase function are quasihomogeneous or almost quasihomogeneous,
then in spite of Mather-Yao type theorems these singularities are determined by the fiber germs of the 
logarithmic Gauss map, because $Q_\varphi=T_\varphi$ or $Q_\varphi/(Ann(\bf{m}))=T_\varphi$, respectively.   
Note, that all simple or unimodal critical points belong to these singularities.
The singularities of a phase function on their part determine the 
asymptotic of corresponding oscillating intergrals.  

All degenerated critical points are lying over the singularities of the 
discriminant ${\mathcal D}\subset
\mathbb{P}^{n-1}$ of $\gamma_f$. Many results could be found concerning the 
connection between singularites of discrimant and singularities in the fiber. 
We try to discuss some 
consequnces with respect to our setting. 
 
The finite map $\overline{\gamma}_f$ can be considered as family over $\mathbb{P}^{n-1}$ 
of complete
intersections (of relative dimension zero). Let $(X_0,0)$ be a germ of an 
isolated complete intersection singularity, let 
 $X\rightarrow S$ its versal family with
discriminant $D\subset S$, the singularity of the discriminat $(D,0)$ determines 
the special fiber
$(X_0,0)$ up to isomorphy by a result of Wirthmuller, c.f \cite{Wirt}. If $\dim (X_0)=0$ 
the multiplicity of the discriminant fulfills 
$\textup{mult}(D,0)=\dim_\mathbb{C}({\mathcal O}_{X_0})-1=\dim_\mathbb{C}({\mathcal O}_{Sing(X_0)})$, 
as a consequnce of \cite{Le}, for instance. 
This is globalized straight forward.
\begin{pro}
Let $\gamma:X\rightarrow S$ be a finite morphism with discriminant $D\subset S$ and
each $X_s$ is a complete intersection, then holds:
$$\textup{mult}_s(D)\geq \sum_{z_i\in X_s}\textup{mult}(Sing(X_s),z_i)=\sum_{z_i\in X_s}(\textup{mult}(X_s,z_i)-1).$$ 
Moreover, equality holds at $s\in S$, if $\gamma$
induces a versal deformation
of $X_s$. 
\end{pro}
\begin{proof}
The local branches of $D$ at $s$ are corresponding to the discriminants $D_i$ of the germs $(X,z_i)
\rightarrow (D,s)$, hence the multiplicities of $D_i$ add up to the multiplicity of $D$. Any  
family  is locally induced from a versal one, hence the discriminant is induced by base chance from 
the discriminant of the versal family and its multiplicity cannot become smaller.
\end{proof}
Note, versality is an open property and corresponds to some kind of {\em stability} in the sence of Mather.
It is not clear for us, whether (or under which additional assumtions) 
the logarithmic Gauss map $\gamma_f$ for a generic function $f(z)$ with fixed Newton polyhedron ${\mathcal N}$
has this stability property. It holds in all computed examples. But, an answer needs further investigation.

Inspecting the classification of hypersurface singularities we get the types of possible critical points
for small multiplicities of the discriminant, which are listed in the following Corollary.
\begin{cor}
Given a Laurent polynomial $f(z)$, non-singular with respect to its Newton polyhedron, and let $\gamma$ be the 
corresponding logarithmic Gauss map with discriminant $D\subset \mathbb{P}^{n-1}$.
Let $m=m(q):=\textup{mult}(D,q)$, then the following configurations are met for the fiber $F_q:=\gamma_f^{-1}(q)$, respectively 
for the collection of critical points of the phase function $\varphi_q(z)$:
\begin{itemize}
\item $m=1$: $F_q$ has exactly one point $z_*$ of multiplicity 2,
$\varphi_q$ has non-degenerated critical point and one 
$A_2$-singularity at $z_*$. 
\item $m=2$: $F_q$ either one point of multiplicity 3 or at most two points of multiplicity 2, 
$\varphi_q$ has at most one $A_3$ or two $A_2$-points.
\item $m=3$: Besides $A_1$ can occur the following collections of critical points of $\varphi_q$:
one $D_4$ or one $A_4$ or a combination $k_2A_2+k_3A_3$ with
$k_2+2k_3\leq 3$.
\item $m\leq 6$: Type of critical set of $\varphi_q$: Only (simple) ADE-critical points can occur 
$$\sum_{i\geq 1} k_iA_i+\sum_{i \geq 4} l_iD_i+\sum_{i=6}^8 n_iE_i,$$ 
such that
 $$\sum i(k_i+l_i +n_i) \leq n!\,vol({\mathcal N})$$ and 
 $$\sum (i-1)(k_i+l_i+n_i)\leq m.$$
\item $m \leq 6$: all critical points are quasihomogeneous (and simple or unimodal).
\item $m \leq 14$: all critical points are almost quasihomogeneous (and simple or unimodal).
\end{itemize}
\end{cor}  
The first critical point, which is not almost quasihomogeneous are the bimodal exceptional 
singularities with smallest Milnor number $\mu=16$ of type $Q_{16}$ or $U_{16}$, cf.\cite{AVGZ}.
 They can occur only at multiplicity $m\geq 15$.

\section{Representation of Diagonal Coefficient by Oscillating Integrals and its Phase Function}

In this section we return to Laurent series~(\ref{eq:1}) converging in $\Log^{-1}(E_\nu)$.
We explane the residue asymptotics formula for its diagonal 
coefficient in the direction $q\in  \mathbb{Z}^n\cap K_\nu$.

Recall, that the Laurent series coefficient can be represented in the form
$$
c_\alpha^\nu=\frac{1}{(2\pi\imath)^n}\int_{\Gamma_\nu}\,\frac{\omega}{z^{\alpha+\bf{1}}},
$$
where $\omega:=F(z)dz$ and the cycle $\Gamma_\nu$ is $n$-dimensional real torus 
$\textup{Log}^{-1}(x_\nu), \ x_\nu\in E_\nu$.  
The direction $q$  induces series of diagonal coefficients
\begin{equation}
\label{eq:2a}
c_{q\cdot k}^\nu=\frac{1}{(2\pi\imath)^n}\int_{\Gamma_\nu}\,\frac{\omega}{z^{q\cdot k+\bf{1}}}.
\end{equation}
We may assume that the point $x_\nu$ generates a line $L:=\mathbb{R}x_\nu\subset\mathbb{R}^n$, 
which is transversal
to the boundary $\partial E_\nu$  and intersects it at a point $p$ and the normal vector at $p$
to $\partial E_\nu$ coincides with the vector~$q$. 
In other words, $p$ is the $\Log$-image of  points
$w^{(1)}(q),\ldots, w^{(N)}(q)$ from the fiber $\gamma^{-1}(q)$ of the logarithmic Gauss mapping.
The torus
$\Log^{-1}(p)\subset\Log^{-1}(L)$ intersects the hypersurface $V^*$ 
at most in $N\leq\textup{Vol}(\mathcal{N}_f)\cdot n!$ points.

Consider a heighbourhood $U_i$ in $\mathbb{C}^n$ of the point $w^{(i)}(q)$, 
then $\Log^{-1}(L)$ intersects the hypersurface $V^*$ in $U_i$ along an 
$(n-1)$-dimensional chain $h_i\subset V^*$.
It can be shown, cf.~\cite{Ts2} for the case $n=2$, that integral~(\ref{eq:2a})
is asymptotically equivalent for $k\to +\infty$ to the sum 
\begin{equation}
	\label{eq:4}
		c_{q\cdot k}^{\nu}=
		\frac{1}{(2\pi\imath)^{n-1}}\sum_{i=1}^{N}
		\int_{h_i}\,\textup{res}\,\left(\frac{\omega}{z}\right)\cdot e^{-\left<q,\log z\right>\cdot k},
\end{equation}
where $\log z=(\log z_1,\ldots, \log z_n)$, $\textup{res}\, (\omega/z)$ is the residue form.
In local coordinates $z'=(z_1,\ldots, z_{n-1})$ of $V^*$ (assuming $f_{z_n}\neq 0$) we have
$\,\textup{res}\,(\frac{\omega}{z}) =\frac{g dz'}{z'  \left. f_{z_n} \right|_{V^{*}}}$. 
Therefore, the diagonal coefficient can be represented as the sum of oscillating integrals with the
 {\em phase function} $\varphi_q(z')=\left.\left<q, \log z\right>\right|_{V^*}$. 
The critical points of this phase function give the main contribution 
to the asymptotic of such integrals. From Proposition~3 follows, that the support
of $h_i$ contains only one critical point of $\varphi_q,$ it is a point $w^{(i)}(q)\in \gamma^{-1}(q)$.

The asymptotics of an oscillating integral is most simple for  Morse critical points. 
In this case it is given by stationary phase method (also called saddle-point method, see~\cite{Wong}).
The Corollary~1 of Proposition~4 states, that for directions $y=q$ outside the ramification locus of
the logarithmic Gauss map $\gamma$ the phase function $\varphi_q$ has only Morse critical points.

The situation of a degenerated critical point is much more complicated. 
First of all we are looking only for rational critical points!
By a result of Varchenko
some information about asymptotics of oscillating integral can be read from 
the distance of the Newton diagramm  of the phase function  at the corresponding point 
in case of a Newton non-degenerated phase function  (and then it depends only of the 
${\mathcal K}$-equivanence class of the hypersurface singularity). Otherwise, the distance is only a lower bound. 
So called {\em adapted coordinates} exist always in dimension 2, such that the phase function is 
Newton non-degenerated. Adapted coordinates can be computed algorithmically, for more details 
cf. \cite{Va} and \cite{IM}. 

\section{Discussion of examples}

\begin{example}

Consider the smooth hypersurface $V^*(f)$ defined as a zero set of the polynomial
$$f=z_1^2 z_2 +z_1 z_2^2 -z_1 z_2 +a,\ a\in\mathbb{R}, \ a\not=0,\frac{1}{27},$$
which is non-degenerated for its Newton polyhedron. The cubic $V^*(f)$ is 
a two-dimensional real torus with three removed points.

The solutions $z(y)=(z_1(y), z_2(y))$ of 
\begin{equation}
\label{eq:ex1}
\left\{
\begin{array}{r}
z_1^2 z_2 +z_1 z_2^2 -z_1 z_2 +a=0\\
h:=(2y_2 -y_1) z_1^2 z_2 +(y_2-2y_1) z_1z_2^2+(y_1-y_2)z_1 z_2=0\\
\end{array}
\right.
\end{equation}
for fixed parameter $(y_1:y_2)\in \mathbb{P}^1$ are zeroes of ideal~(\ref{eq:cv}) and for real
parameter $y$ they are 
projected to the contour $\mathcal{C}_{V^*}$
by $\textup{Log}$-map. We are interested in the real ramification locus of $\gamma_f$. 

We compute the resultant of $f,h$ with respect to the variable $z_2$ 
$$
\begin{array}{ll}
Res(f, h):= & (-y_1^2+y_1 y_2 +2y_2^2) z_1^3+(2y_1^2-2y_1y_2-y_2^2)z_1^2\\
\ &  +(-y_1^2+y_1 y_2)z_1+ 4a y_1^2  -4a y_1 y_2+a y_2^2.
\end{array}
$$
The multiplicity of  an isolated zero $z(y)$ of system~(\ref{eq:ex1}) coincides with the
multiplicity of the zero $z_1(y)$ in $Res(f,h)$. The discriminant  
of the polynomial $Res(f,h)$ with respect to $z_1$ 
is the homogeneous polynomial 
$$
\begin{array}{ll}
	D(y_1, y_2) &=(1-27a)(-2y_1+y_2)^2 ( 4ay_1^6-12a y_1^5 y_2 +(-3a+1)y_1^4y_2^2- \\
	\ &  -2(1-13a)y_1^3 y_2^3 + (-3a+1) y_1^2 y_2^4 -12a y_1 y_2^5+ 4a y_2^6)
\end{array}
$$
in variables $y_1,y_2$. 

Interested in roots of (\ref{eq:ex1}) in $\mathbb{T}^2$ we can omit the factor $(-2y_1+y_2)^2$ in the last expression.
Substituting in $D(y)$  an affine parameter $\lambda=y_1/y_2$ we get the polynomial
$$
D(\lambda)= 4a\lambda^6-12a \lambda^5 +(-3a+1)\lambda^4 -2
	(1-13a)\lambda^3 + (-3a+1) \lambda^2 -12a\lambda+ 4a,
$$
whose real zeroes $\lambda_i$ give the  points $(\lambda_i:1)\in \mathbb{P}^1_{\mathbb{R}}$ 
of the real ramification locus of $\gamma_f$.

We have three real intervals of the paramter line $\mathbb{R}_a$:
for $a<0$ the polynomial $D(\lambda)$ has
six real roots, for $0<a<\frac{1}{27}$ and $\frac{1}{27}<a$
the polynomial $D(\lambda)$
has no real roots.

Choosing values of $a$ from the different intervals of $\mathbb{R}_a$ we obtain different configurations
of the contour $\mathcal{C}_{V^*}$ and the amoeba $\mathcal{A}_{V^*}$. Because the volume 
$2!\cdot\textup{vol}(\mathcal{N}_f)=3$
does not depend on $a$, all these configurations have a following common property:
the  number of preimages $\textup{Log}^{-1}(p)$ of a point $p\in \mathcal{C}_{V^*}$
with normal vector $(y_1, y_2)\in \mathbb{R}^2$ is equal to three, provided we count such preimages, which
are solutions to~(\ref{eq:ex1}) for corresponding $(y_1:y_2)\in \mathbb{P}^1_{\mathbb{R}}$,
with their multiplicity in~(\ref{eq:ex1}). Hence, one can find for every $\lambda\in \mathbb{R}$
three points on $\mathcal{C}_{V^*}$ with the normal vector $(\lambda, 1)$. Moreover, each one lies on its own
colored or black part of the contour (see. Fig.~1).

On the left Fig.~1 six black points on $\mathcal{C}_{V^*}$ are images of pleat singularities of the mapping 
$\left.\textup{Log}\right|_{V^*}$,
they correspond to values $\lambda_i$ that belong to the real ramification locus of $\gamma_f$.

Although, for $a>0$ the real ramification locus of $\gamma_f$ is empty, we can distinguish two situations.
If $0<a<1/27$ the hypersurface $V^*(f)$  is a complexification of the so-called Harnack curve 
and the complement of its amoeba has the maximal number 
of components. In this case $\left.\textup{Log} \right|_{V^*}$ has only fold singularities, 
which coincide with $V^*(f)\cap \mathbb{R}^2$. 
If $a>1/27$ the complement of the amoeba $\mathcal{A}_{V^*}$ has no bounded component, and the mapping 
$\left.\textup{Log}\right|_{V^*}$ has  three pleat singularities and other singularities are folds.

\begin{center}
\includegraphics{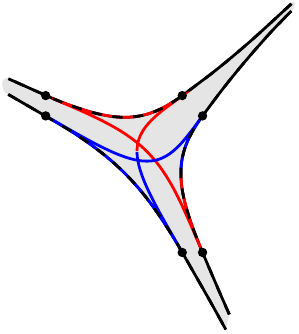}\hskip .25cm
\includegraphics{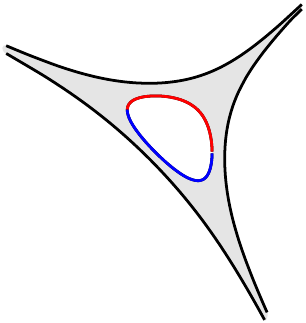}\hskip .25cm
\includegraphics{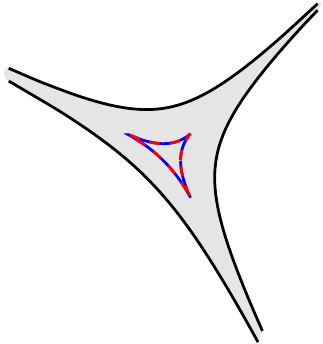}\\
\bigskip

{\small Fig.~1. The contour and the amoeba (shaded) for the polynomial $f=z_1^2 z_2 +z_1 z_2^2 -z_1 z_2 +a$:\\ on the left $a<0$, in the middle $0<a<1/27$,
	on the right $a>1/27$.}
\end{center}

Therefore, for $a>0$ $\gamma_f$-fiber of any rational  $\lambda$ contains only the Morse critical points of the phase function.
For example, set the parameter $a=3/100$ then the $\gamma_f$-fiber of $\lambda=1/3$ consists of the Morse points
$(3/10, 1/2)$, $(7/40+\sqrt{57}/40, 9/8-\sqrt{57}/8)$ and $(7/40-\sqrt{57}/40, 9/8 +\sqrt{57}/8)$. 
For $a<0$ we can get degenerated rational points in a real ramification locus, e.g. there are six rational points 
$-2, -1/2, 2/3, 3/2, 1/3,3$ in the real ramification locus of $\gamma_f,$ $a=-9/10.$ The  $\gamma_f$-fiber of such points
has a simple point and an $A_2$-point of the phase function.

\end{example}

\bigskip
\begin{example}
We consider the polynomial $f$ in $n=3$ variables and non-degenerated for its Newton polyhedron,
$$
	f=1+z_1+z_2+z_3+3z_1 z_2 +3 z_1 z_3 + 3 z_2 z_3+11 z_1 z_2 z_3.
$$
As in Example~1 the real ramification locus  of $\gamma_f$ is determined by  the 
following system
\begin{equation}
\label{eq:ex2}
\left\{
\begin{array}{r}
1+z_1+z_2+z_3+3z_1 z_2 +3 z_1 z_3 + 3 z_2 z_3+11 z_1 z_2 z_3  =0,\\
y_3 z_1 - y_1 z_3 + 3y_3 z_1 z_2 +(3y_3-3 y_1)z_1 z_3 -3y_1 z_2 z_3 \hskip .75cm \\
+(11y_3 -11 y_1) z_1 z_2 z_3 =0,\\
y_3 z_2 - y_2 z_3 + 3y_3 z_1 z_2 -3y_2 z_1 z_3+(-3y_2+3y_3)z_2 z_3 \hskip .75cm  \\
+(11y_3 -11 y_2) z_1 z_2 z_3 =0.\\
\end{array}
\right.
\end{equation}
With similar computations we obtain the discriminant $D(y)$ of the logarithmic Gauss map.

\begin{equation}
\label{eq:ex22}
	D(y):=y_1^4\cdot(y_2-y_3)^2\cdot (4y_1+5y_2+5y_3)^2\cdot d(y),
\end{equation}
here $d(y)$ is a homogeneous polynomial of degree~$12$, it consists of 91 terms. Its Newton's
polyhedron is a triangle with vertices $(12,0,0)$, $(0,12,0)$ and $(0,0,12)$.

We do not consider zeroes of the first three factors in~(\ref{eq:ex22}), because
they do not give us multiple roots of~(\ref{eq:ex2}) in the torus. 
The ramification locus of $\gamma_f$
is given by zero set of $d(y)$. Let $\lambda_1=y_1/y_3$, $\lambda_2=y_2/y_3	$ 
be coordinates in affine part of $\mathbb{P}^{2}_{\mathbb{R}}$,  where $y_3\neq 1$. 
Fig.~2 shows the zero set of $d(\lambda_1,\lambda_2,1)=d(y)/y_3^{12}$, which 
coincides with the affine part of the real ramification locus of $\gamma_f$.

\begin{center}
\includegraphics{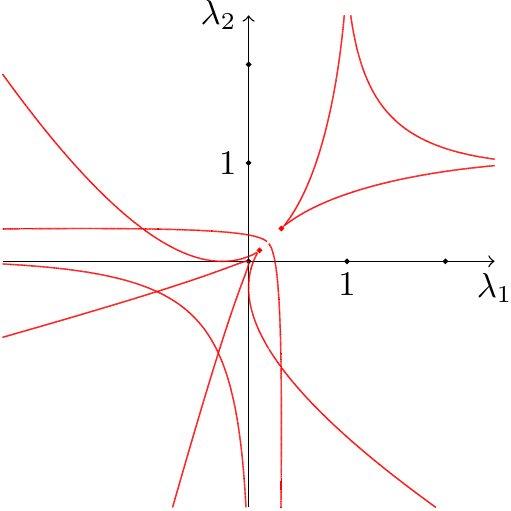}
\includegraphics{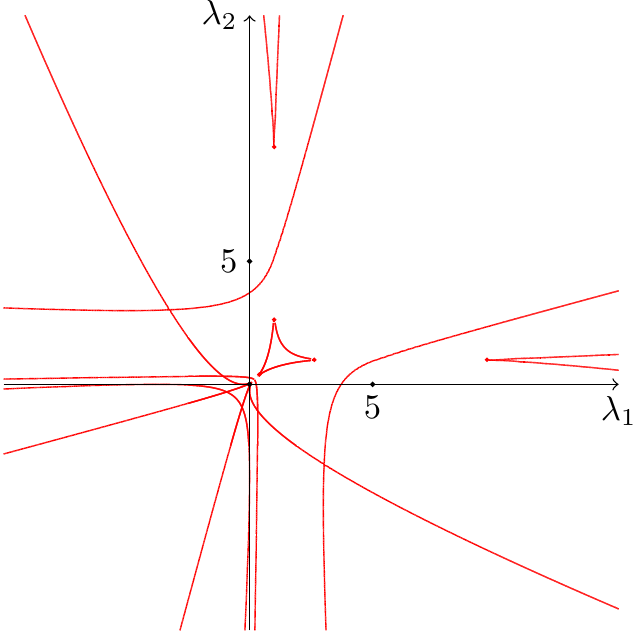}

\bigskip

{\small Fig.~2. The real ramification locus of $\gamma_f$.}
\end{center}

The red points $(1/9,1/9),\ (1/3,1/3),\ (1,3),\ (3,1),\ (1,9),\ (9,1)$ on Fig.~2 are  
the degerated rational critical points 
of the discriminant with Milnor number $\mu=2$. 
This example of the polynomial $f$ is a special one, 
because the existence of rational degenerated points in a real ramification locus is 
not a generic property. But, we are interested in such points, because they 
lead to  degenerated critical points of a phase function. 
In this example the $\gamma_f$-fiber of any $A_2$-point  
contains excately one $A_3$-critical point of the phase function (see Appendix for details).

\end{example}

\mbox{}\\[1mm]
{\Large \bf Appendix: \\ Computation with {\sc Singular} (some experiences)}\\[1mm]
The computer algebra system {\sc Singular}, cf. \cite{sing}, was used for the computation of 
examples. We tried several strategies for computing the discriminant of the Log-Gauss map with
different success, i.e. to get a result for non-trivial examples without 
overflow and in reasonable time. Here we give a small introduction how proceed in {\sc Singular},
demonstrated with the equation of example 2.

Start with a base ring that contains the ideal $I$ of the graph of the log Gauss map $\gamma_f$ of a
polynomial $f=f(z)$ and compute $I$, (here $n=3$):
\begin{quote}
{\tt ring R=0,(y1,y2,y3,z1,z2,z3),dp; \\
     poly f=1+z1+z2+z3+3*z1*z2 +3*z1*z3+3*z2*z3+11*z1*z2*z3;\\
     matrix A[2][3] = z1*diff(f,z1),z2*diff(f,z2),z3*diff(f,z3),y1,y2,y3;\\
     ideal I = f,minor(A,2);
}     
\end{quote}
Next we project the graph restricted to some affine chart
$U_3:=\{y_3\not=0\}$ into $\mathbb{A}^3:=U_3\times \mathbb{A}^1$ ($\mathbb{A}^1_1$ - a
coordinate axes of $\mathbb{A}^3_z$). The image is a hypersurface defined by a
polynomial $h(y_1,y_2,z_1)$, which we could closure in $\mathbb{P}^2_y$ 
by homogenizing in the $y's$. 
Using the elimination of variable, 
the multiplicities of multiple factors may be lost, but it does not effect 
the result.
\\
\begin{quote}
{\tt 
   I = subst(I,y3,1);
   ideal J = eliminate(I,z2*z3);\\
   poly h1 = J[1];  
}     
\end{quote}

The choice of the projection direction was good, if $deg_{z1}(h_1)=n!\,vol({\mathcal N})=6$.
The discriminant variety of $\gamma_f$ is contained in the discriminant hypersurface 
of the projection $V(h_1)\subset U_3\times \mathbb{A}^1_1\longrightarrow  U_3$,
computed in the next step. 
\\
\begin{quote}
{\tt 
   poly d1 = resultant(h1,diff(h1,z1),z1);\\
   d1 = homog(d1,y3);\\
   list Ld = factorize(d1);
}     
\end{quote}

The plane curve $V(d_1)\subset \mathbb{P}^2$ has several components, it may have components 
with certain multiplicities, some of them induced from the closure $V^*(f)$ or 
not belonging to the discriminant. If our polynomial is generic, then we expect the discriminant 
of $\gamma_f$ (i.e. restricted to the torus) being irreducible. We should test which factor is correct.
Some components of $V(d_1)$ have empty fiber with respect to $\gamma_f$ or no multiple points in its
$\gamma_f$-fibers.
We can reduce sometimes the number of factor as follows: Compute for any coordinate $z_i$ (as above for $i=1$)
polynomials $h_i$ and $d_i$ and factorize only  $d:=gcd(d_1,\ldots ,d_n)$. 
\\
Having found the equation of 
the discriminant polynomial $d_0(y)$, we can compute  its (discrete) singular locus.
\\
\begin{quote}
{\tt 
   poly d0 = Ld[1][2];\hspace{2cm}\mbox{\it (choose the right factor in this example)}\\
   d0  = subst(d0,y3,1); \\
   ring S = 0,(y1,y2,y3),dp;\\
   poly d0 = imap(R,d0);\\
   ideal sl = slocus(d0);\\
   list Lsl = primdecGTZ(sl);
 }     
\end{quote}  

Here, the singular locus has six rational double points $Q_1=(1,3)$, 
$Q_2=(\frac{1}{9},\frac{1}{9})$, $Q_3=(1,9)$, $Q_4=(\frac{1}{3},\frac{1}{3})$, 
$Q_5=(9,1)$, $Q_6=(3,1)$
  and more irrational singular points.
We choose $Q_1$ and check, that it is an $A_2$-singularity of $D$.
 \\
\begin{quote}
{\tt 
   show(Lsl[2]);\hspace{26mm}\mbox{\it (choose one of the singular points of $D$)}\\
   ring S' = 0,(y1,y2),ds;\\
   poly d0 = imap(R,d0);\\
   d0 = subst(d0,y1,y1+1); \hspace{7mm}\mbox{\it (translate that singularity to zero)}\\
   d0 = subst(d0,y2,y2+3);\\
   "mu =",milnor(d0);\hspace{2cm}\mbox{\it (Milnor number of the singularity)}\\
 }     
\end{quote}  

Compute the $\gamma_f$-fiber of $Q_1$. Its has 3 simple points and exactly one point 
$P_*=(-1,-\frac{1}{3},-1)$ of
multiplicity 3, being an $A_3$-point of the phase function.    
\\
\begin{quote}
{\tt    
   setring R;\\
   I = subst(I,y1,1); I = subst(I,y2,3);\\
   ring R0 = 0,(z1,z2,z3),dp;\\
   ideal I = imap(R,I);\\
   list Lfib = primdecGTZ(i);\hspace{2cm}\mbox{\it (list contains the points of the fiber)}.\\
   option(redSB);\\
   show(std(Lfib[1][2]));\\
   "mult =",vdim(std(Lfib[1][1]));  
}       
\end{quote}
Similar computations lead to similar results at the other 5 rational singularities of
the discriminant.


\begin{thebibliography}{99}
 
\bibitem{AVGZ}
Arnold~V.I., Gusein-Zade~S.M., Varchenko~A.N. Singularities of differentiable
maps Vol. I, Birkhauser, 1985.

\bibitem{sing}
Decker, W.; Greuel, G.-M.; Pfister, G.; Sch{\"o}nemann, H.: 
\newblock {\sc Singular} {3-1-3} --- {A} computer algebra system for polynomial computations.
\newblock {http://www.singular.uni-kl.de} (2011).



\bibitem{FPT}
Forsberg~M., Passare~M., Tsikh~A. Laurent determinants and
arrangements of hyperplane amoebas
//~Adv. in Math. 2000. \textbf{151}. p.~45-70.


\bibitem{Gelfand}
Gelfand~I, Kapranov~M., Zelevinsky~A. Discriminants, Resultants
and Multidimentional Determinants. Boston: Bikh\"auser,~1994.

\bibitem{IM}
Ikromov~I.A., M\"{u}ller D. On adapted coordinate system. Trans. Amer. Math. Soc. 2011, \textbf{363}. p.~2821-2848.


\bibitem{Ho}
Khovanskii A.G. Newton polyhedra and toroidal varieties 
// Funct. Anal. Appl. 1977. \textbf{11}:4. p.~289-296.

\bibitem{Kou}
Kouchnirenko A.G. Poly\`{e}dres de Newton et nombres de Milnor
// Invent. Math. 1976. \textbf{32}. p.~1-31.

\bibitem{Le}
Le~D.T. Calculation of Milnor number of isolated singularity of complete intersection//
Funk.Anal. and its Appl. 1972.  \textbf{8}:2, p.~45-49.


\bibitem{Yau} 
Mather~J.N. Classification of Isolated Hypersurface Singularities by Their Moduli Algebras//
Invet.math. 1982. \textbf{69}, p.~243-251.

\bibitem{BM}
Martin~B. Singularities are determined by their cotangent complexes//
Ann. Global Anal. Geom. 1985. \textbf{3}:2, p.~197-217.



\bibitem{Mi1}
Mikhalkin~G. Real algebraic curves, moment map and amoebas 
// Ann. of Math. 2000. \textbf{151}:1. pp.~309-326.

\bibitem{PT}
Passare~M., Tsikh~A.K. Amoebas: their spines and their contours 
// Contemporary Math. 2005. \textbf{377}. p.~275-288.


\bibitem{Ru}
Rullg\o{a}rd~H. Topics in geometry, analysis and inverse problems. Ph. D
diss. Stockholm University. 2003. ISBN 91-7265-738-3.


\bibitem{Tes}
Teissier~B. The hunting of invariants in the geometry of discriminants, Real and complexe singularities, Oslo 1076,
pp.~565-678.

\bibitem{Ts2}
Tsikh~A.K. Conditions for absolute convergence of the Taylor coefficient series
of a meromorphic fucntion of two variables 
// Math. USSR SB. 1993, \textbf{74}:2. p.~337-360.


\bibitem{Va}
Varchenko~A. N. Newton polyhedra and estimates of oscillating integrals
// Funct. Anal. Appl. 1976. \textbf{10}:3. p.~175-196.


\bibitem{Wirt} 
Wirthm\"uller, K. Singularities Determined by Their Discriminant//
Math.Ana. 1980. \textbf{252}, p.~237-245.

\bibitem{Wong}
Wong~R. Asymptotic approximations of integrals. SIAM, 2001.




%






\end{thebibliography}
\end{document}